\documentclass{amsart}
\usepackage[T1]{fontenc}

\newtheorem{theorem}{Theorem}[section]
\newtheorem{lemma}[theorem]{Lemma}

\newtheorem{example}[theorem]{Example}
\newtheorem{question}[theorem]{Question}
\theoremstyle{definition}

\numberwithin{equation}{section}

\begin{document}

\title[Gen. Inverse Limits Indexed by Totally Ordered Sets]%
{Generalized Inverse Limits \\ Indexed by Totally Ordered Sets}

\author{Scott Varagona}
\address{Department of Biology, Chemistry \& Mathematics; University of Montevallo;
Montevallo, Alabama 35115}
\email{svaragona@montevallo.edu}
\thanks{}

\subjclass[2010]{Primary 54F15, 54H20}

\keywords{inverse limits, upper semi-continuous inverse limits, connected, set valued function, totally ordered sets, Continuum Theory}
\thanks {The author is indebted to W. T. Ingram and William Mahavier for their groundbreaking work on this topic (in \cite{ingram mahavier}), to Michel Smith, who first introduced the author to ``long'' inverse limits, and to Van Nall, Sina Greenwood, and Steven Clontz for their feedback on the presentation this paper is based upon.}

\begin{abstract} Although inverse limits with factor spaces indexed by the positive integers are most commonly studied, Ingram and Mahavier have defined inverse limits with set-valued functions broadly enough for any directed index set to be used. In this paper, we investigate generalized inverse limits whose factor spaces are indexed by totally ordered sets. Using information about the projections of such inverse limits onto finitely many coordinates, we generalize various well-known theorems on connectedness in inverse limits. Moreover, numerous theorems and examples are given addressing the special case of an inverse limit with a single idempotent surjective u.s.c. bonding function.
\end{abstract}

\maketitle

\section{\bf Introduction}
In recent years, the vast majority of work on generalized inverse limits has focused only on the case where the factor spaces are indexed by the positive integers. However, in \cite{ingram mahavier}, W. T. Ingram and William S. Mahavier define generalized inverse limits much more broadly, so that any directed index set could be used---for example, an uncountable totally ordered set such as the limit ordinal $\omega_1$. In the case of traditional inverse limits with continuous bonding functions, such ``long'' inverse limits have proven to be fruitful for continuum theory; for example, Michel Smith \cite{smith} and later David Bellamy \cite{bellamy} used long traditional inverse limits to construct non-metric indecomposable continua with remarkable properties. It is therefore natural for us to investigate long generalized inverse limits as well.

\
After the groundbreaking work of Ingram and Mahavier, some other researchers have worked with generalized inverse limits using alternate index sets. Inverse limits indexed by the set of all integers have been investigated by the author \cite{varagona1} and, to a more significant degree, by Patrick Vernon \cite{vernon}; theorems about Mahavier Products indexed by totally ordered sets have been proven by Wlodzimierz Charatonik and Robert Roe in \cite{char roe}. However, in general, relatively little is known about generalized inverse limits indexed by any set aside from the positive integers.

\
The goal of this paper is to expand the theory of generalized inverse limits indexed by totally ordered sets. To that end, we present a variety of theorems and examples that should be a good starting point for future research. After giving basic definitions and background information in Section 2, in Section 3 we discuss the behavior of the projections of inverse limits indexed by totally ordered sets. In Section 4, we study the properties of idempotent u.s.c. functions; such functions, it seems, will be important to this type of inverse limit. Next, in Section 5, we further generalize some of the connectedness results by Ingram and Mahavier in \cite{ingram mahavier}. We also extend some of the connectedness results by Van Nall in \cite{nall connected} to the case of generalized inverse limits indexed by totally ordered sets, at least in the case of inverse limits with a single idempotent u.s.c. bonding function $f$. Then, in Section 6, we apply our results to study examples of inverse limits indexed by some limit ordinal $\gamma \ge \omega$. Just as inverse limits indexed by the positive integers have been useful for representing complicated metric continua in a straightforward way, inverse limits indexed by ``long'' initial segments of the ordinals can be useful in the same way for certain non-metric continua. Finally, we conclude by stating some questions for future investigation in Section 7.

\section{\bf Definitions and Background Theorems}

If $X$ is a non-empty compact Hausdorff space, let $2^{X}$ denote the set of non-empty compact subsets of $X$. $C(X)$ denotes the set of connected members of $2^{X}$. Suppose both $X$ and $Y$ are non-empty compact Hausdorff spaces and $f: X \rightarrow 2^{Y}$ is a set-valued function; then we say $f$ is \emph{upper semi-continuous} (u.s.c.) if, for all $x \in X$, whenever $V$ is open in $Y$ containing $f(x)$, there exists an open $U$ in $X$ containing $x$ so that $f(u) \subseteq V$ for each $u \in U$. The \emph{graph} of $f$, denoted in this paper by $Graph(f)$, is the set $\{(x,y) \in X \times Y \ | \ y \in f(x)\}$; it is well-known that $f$ is u.s.c. iff $Graph(f)$ is a closed subset of $X \times Y$. If $f(x) = \{y\}$ for some $x \in X$, $y \in Y$, we simply write $f(x) = y$. Given some u.s.c. function $f: X \rightarrow 2^{Y}$, the \emph{preimage via $f$} of a given $y \in Y$ is $f^{-1}(y) = \{x \in X \ | \ y \in f(x)\}$. (It is not hard to see that $f^{-1}(y)$ is a compact subset of $X$.) If $A \subseteq X$, then $f(A) =  \bigcup_{a \in A} f(a)$. Similarly, if $B \subseteq Y$, $f^{-1}(B) = \bigcup_{b \in B} f^{-1}(b)$. We say $f$ is \emph{surjective} if, for all $y \in Y$, $f^{-1}(y)$ is non-empty. If $f: X \rightarrow 2^{Y}$ is surjective, then the \emph{inverse of $f$}, denoted $f^{-1}$, is the u.s.c. function $f^{-1}: Y \rightarrow 2^{X}$ whose graph is $Graph(f^{-1}) = \{(y,x) \ | \ (x,y) \in Graph(f)\}$.

\
Given non-empty compact Hausdorff spaces $X, Y$ and $Z$ and u.s.c. functions $f: X \rightarrow 2^{Y}$ and $g: Y \rightarrow 2^{Z}$, the composition $g \circ f : X \rightarrow 2^{Z}$ is the u.s.c. function given by $(g \circ f)(x) = \{z \in Z \ | \ \exists y \in Y$ such that $y \in f(x)$ and $z \in g(y)\}$. In the special case where $f: X \rightarrow 2^{X}$ is u.s.c., the composition $f \circ f$ is denoted $f^2$. When $f^2 = f$, we say $f$ is \emph{idempotent}.

\
If $X_{\alpha}$ is a compact Hausdorff space for each $\alpha$ in an index set $A$, then $\prod_{\alpha \in A}X_{\alpha}$ denotes the product space with the usual product topology. If $B \subseteq A$, then $\pi_{B}$ denotes the projection map from $\prod_{\alpha \in A}X_{\alpha}$ into $\prod_{\alpha \in B}X_{\alpha}$. If $\beta \in A$, we will write $\pi_{\{\beta\}}$ as $\pi_{\beta}$. Let us use a boldface $\textbf{x}$ to denote an element $(x_{\alpha})_{\alpha \in A}$ of the product $\prod_{\alpha \in A}X_{\alpha}$. (So, if $\beta \in A$, $\pi_{\beta}(\textbf{x}) = x_{\beta}$.)

\
A \emph{directed set} $(D, \preceq)$ is a set $D$ together with a relation $\preceq$ on $D$ that is reflexive, transitive, and has the property that whenever $\alpha, \beta \in D$, there exists some $\eta \in D$ such that $\alpha \preceq \eta$ and $\beta \preceq \eta$. A directed set $(D, \preceq)$ will often be denoted simply by $D$ when the relation $\preceq$ is understood. A directed set $D$ is \emph{totally ordered} if, $\forall \alpha, \beta \in D$, either 1) $\alpha \preceq \beta$ and $\beta \not\preceq \alpha$, 2) $\beta \preceq \alpha$ and $\alpha \not\preceq \beta$, or 3) $\alpha \preceq \beta$ and $\beta \preceq \alpha$, in which case $\alpha = \beta$. Whenever we write a finite subset $\{\beta_1, \beta_2, \ldots, \beta_{n}\}$ of some totally ordered set $D$, we tacitly assume that $\beta_{i} \preceq \beta_{j}$ in the ordering on $D$ iff $i \le j$ in the standard ordering on the natural numbers. Also, if $\alpha \preceq \beta$ but $\alpha \ne \beta$, then we write $\alpha \prec \beta$.

\
The following notation, and indeed, most of the notation used in this paper, is intended to coordinate with the notation used in \cite{ingram mahavier}. Suppose that, for each element $\alpha$ of a directed set $D$, $X_{\alpha}$ is a non-empty compact Hausdorff space. Moreover, if $\alpha, \beta \in D$ with $\alpha \preceq \beta$, let $f_{\alpha \beta} : X_{\beta} \rightarrow 2^{X_{\alpha}}$ be u.s.c. (where $f_{\alpha \alpha}$ always denotes the identity on $X_{\alpha}$). If, for all $\alpha \preceq \beta \preceq \eta$ in $D$, $f_{\alpha \beta} \circ f_{\beta \eta} = f_{\alpha \eta}$, then the collection $\{f_{\alpha \beta} \ | \ \alpha \preceq \beta \in D\}$ is called \emph{exact}. Assuming $X_{\alpha}$ is compact Hausdorff $\forall \alpha \in D$, $f_{\alpha \beta}: X_{\beta} \rightarrow 2^{X_{\alpha}}$ is u.s.c. for all $\alpha \preceq \beta \in D$, and $\textbf{f} = \{f_{\alpha \beta} \ | \ \alpha \preceq \beta \in D\}$ is exact, we say $\{X_{\alpha}, f_{\alpha \beta}, D\}$ is an \emph{inverse limit system} (or simply, a \emph{system}). The \emph{inverse limit}, $\varprojlim \textbf{f}$, of this system is given by $\{ \textbf{x} \in \prod_{\alpha \in D} X_{\alpha} \ | \ x_{\alpha} \in f_{\alpha \beta}(x_{\beta}) \ \forall \alpha \preceq \beta \in D\}$. The spaces $X_{\alpha}$ are called the \emph{factor spaces} of the inverse limit, and the u.s.c. functions $f_{\alpha \beta}$ are called the \emph{bonding functions}. In the special case where $X$ is compact Hausdorff and $f: X \rightarrow 2^{X}$ is an idempotent u.s.c. function, if $X_{\alpha} = X$ for all $\alpha \in D$ and $f_{\alpha \beta} = f$ for all $\alpha \prec \beta \in D$, then $\{X_{\alpha}, f_{\alpha \beta}, D\} = \{X, f, D\}$ is a system and $\varprojlim \textbf{f}$ is an \emph{inverse limit with a single idempotent bonding function f}. (Note that, in this special case, the collection $\textbf{f}$ is automatically exact because $f$ is idempotent.) If $\{X_{\alpha}, f_{\alpha \beta}, D\}$ is a system so that, for each $\eta \in D$ and for each $p \in X_{\eta}$, there exists $\textbf{x} \in \prod_{\alpha \in D}X_{\alpha}$ with $x_{\alpha} \in f_{\alpha \beta}(x_{\beta})$ for all $\alpha \preceq \beta \preceq \eta$ and $x_{\eta} = p$, then we say the system is \emph{consistent}.

\
As we will see, certain sets will be helpful for our study of inverse limits. Suppose $\{X_{\alpha}, f_{\alpha \beta}, D\}$ is a system. If $\eta \in D$, we define $G_{\eta} = \{ \textbf{x} \in \prod_{\alpha \in D} X_{\alpha} \ | \ x_{\alpha} \in f_{\alpha \beta}(x_{\beta})$ for all $\alpha \preceq \beta \preceq \eta\}$. If $\{\beta_{1}, \beta_{2}, \ldots, \beta_{n}\}$ is a finite subset of $D$, then let us say $G(\beta_1, \beta_2, \ldots, \beta_n) = \{(x_{\beta_{1}}, x_{\beta_{2}}, \ldots, x_{\beta_{n}}) \in \prod_{1 \le i \le n}X_{\beta_{i}} \ | \ x_{\beta_{i}} \in f_{\beta_{i} \beta_{j}}(x_{\beta_{j}})$ for $1 \le i \le j \le n\}$. (In the case where two different inverse limits, e.g., $\varprojlim \textbf{f}$ and $\varprojlim \textbf{g}$, with the same index set $D$ are being discussed at the same time, we may use subscripts $G_{f}(\beta_1, \beta_2, \ldots, \beta_n)$ and $G_{g}(\beta_1, \beta_2, \ldots, \beta_n)$ to help distinguish between the two corresponding sets.) We also define $G'(\beta_1, \beta_2, \ldots, \beta_n) = \{\textbf{x} \in \prod_{\alpha \in D}X_{\alpha} \ | \ x_{\beta_{i}} \in f_{\beta_{i} \beta_{j}}(x_{\beta_{j}})$ for $1 \le i \le j \le n\}$. Note that $G(\beta_{1}, \beta_{2}, \ldots, \beta_{n})$ is a subset of $\prod_{1 \le i \le n}X_{\beta_i}$, whereas $G'(\beta_1, \beta_2, \ldots, \beta_n)$ is the analogous subset of $\prod_{\alpha \in D} X_{\alpha}$.

\
Finally, let us say $K$ is a \emph{Hausdorff continuum} if $K$ is a non-empty, compact and connected subset of a Hausdorff space. If $K$ happens to be metrizable, we call $K$ a \emph{metric continuum} or simply a \emph{continuum}. If $H$ is a set and $K$ is a Hausdorff continuum that is a subset of $H$, we call $K$ a \emph{subcontinuum} of $H$. If $X$ and $Y$ are compact Hausdorff spaces and the u.s.c. function $f: X \rightarrow 2^{Y}$ has the property that $f(x)$ is connected for each $x \in X$, then we say $f$ is \emph{Hausdorff continuum-valued} and we write $f: X \rightarrow C(X)$.

\
A great deal of the initial work on generalized inverse limits indexed by totally ordered sets was done by Ingram and Mahavier in \cite{ingram mahavier}. Although, in general, inverse limits indexed by arbitrary directed sets may fail to be non-empty, Ingram and Mahavier showed that the inverse limit of a consistent system with non-empty compact Hausdorff spaces indexed by a directed set $D$ is non-empty and compact (Theorem 111, \cite{ingram mahavier}). Charatonik and Roe \cite{char roe} recently showed that any system indexed by a totally ordered set $D$ is automatically consistent. When combining this result with the original theorems of Ingram and Mahavier, we obtain an important background theorem:

\begin{theorem} Let $\{X_{\alpha}, f_{\alpha \beta}, D\}$ be a system with non-empty compact Hausdorff factor spaces, u.s.c. bonding functions, and a totally ordered index set $D$. Then 1) for each $\eta \in D$, $G_{\eta}$ is non-empty and compact; 2) for each finite subset $\{\beta_1, \beta_2, \ldots, \beta_n\}$ of $D$, $G(\beta_1, \beta_2, \ldots, \beta_n)$ and $G'(\beta_1, \beta_2, \ldots, \beta_n)$ are non-empty and compact; 3) $\varprojlim \textbf{f}$ is non-empty and compact.
\end{theorem}

\
Since $\varprojlim \textbf{f}$ is always non-empty and compact in the context we have described, it is natural to look for conditions under which such an inverse limit would also be connected. The following result (which combines the work of Ingram, Mahavier, Charatonik and Roe) is given as Corollary 2.3 in \cite{char roe} (although, in that paper the result is stated in the wider context of Mahavier Products).

\begin{theorem} \label{background thm} Let $\{X_{\alpha}, f_{\alpha \beta}, D\}$ be a system for which each factor space is a Hausdorff continuum, each bonding function is u.s.c., and $D$ is totally ordered. Suppose that, for each $\alpha, \beta \in D$ with $\alpha \preceq \beta$, either $f_{\alpha \beta}$ is Hausdorff continuum-valued or $f_{\alpha \beta}(X_{\beta})$ is connected with $f^{-1}_{\alpha \beta}(y)$ a Hausdorff continuum for each $y \in f_{\alpha \beta}(X_{\beta})$. Then $\varprojlim \textbf{f}$ is a Hausdorff continuum.

\end{theorem}

\
The above theorem generalizes two important theorems on connectedness in inverse limits given in \cite{i m paper}. Later in this paper, we will generalize other well-known results as well. Whenever possible, we will prove theorems in the broad setting of inverse limit systems $\{X_{\alpha}, f_{\alpha \beta}, D\}$ where each $X_{\alpha}$ is compact Hausdorff, each $f_{\alpha \beta}$ is u.s.c., and $D$ is a totally ordered set. However, we will also give a number of theorems that are particular to the case of an inverse limit with a single idempotent, surjective, u.s.c. bonding function $f$.

\section{\bf Projections of Inverse Limits Onto Finitely Many Coordinates}

To prepare for the main results in the later sections, we first need to discuss the behavior of the projections of an inverse limit. Unfortunately, for a general inverse limit system $\{X_{\alpha},  f_{\alpha \beta}, D\}$, if $H = \{\beta_1, \beta_2, \ldots, \beta_n\}$ is a finite subset of the index set $D$, then $\pi_{H}(\varprojlim f)$ is not necessarily equal to $G(\beta_1, \beta_2, \ldots, \beta_n)$. For example, Ingram and Mahavier give an example (\cite{ingram mahavier}, Example 106) of a system with compact factor spaces, surjective bonding functions, and a (non-totally ordered) directed index set $D$ where $\varprojlim \textbf{f}$ is the empty set, but $G(\alpha, \beta)$ is non-empty for every $\alpha \preceq \beta \in D$. On the other hand, if $f: [0,1] \rightarrow 2^{[0,1]}$ is given by $f(x) = 0$ for $0 \le x \le 1$, then the inverse limit $\varprojlim \textbf{f}$ of the system $\{[0,1], f, \omega\}$ (with the single bonding function $f$) is the singleton $\{(0,0,0,\ldots)\}$. Therefore, in this case, $\pi_{\{0,1\}}(\varprojlim \textbf{f}) = \{(0,0)\}$, whereas $G(0,1)$, being homeomorphic to the graph of $f$, is an arc.

\
Let us say an inverse limit $\varprojlim \textbf{f}$ is \emph{cordial} if, for each finite subset $\{\beta_1, \beta_2, \ldots, \beta_n\}$ of the index set $D$, $\pi_{\{\beta_1, \beta_2, \ldots, \beta_{n}\}}(\varprojlim \textbf{f}) = G(\beta_1, \beta_2, \ldots, \beta_n)$. The goal of this section is to show that an inverse limit (with compact Hausdorff factor spaces indexed by a totally ordered set) is cordial if and only if its bonding functions are all surjective.

\begin{lemma} \label{cordial} Suppose $\{X_{\alpha}, f_{\alpha \beta}, D\}$ is a system with non-empty compact Hausdorff factor spaces, surjective u.s.c. bonding functions and a totally ordered index set $D$. Then $\varprojlim \textbf{f}$ is cordial.
\end{lemma}

\begin{proof} Let a finite subset $H = \{\beta_1, \beta_2, \ldots, \beta_{n}\}$ of $D$ be given. We aim to show that $\pi_{H}(\varprojlim \textbf{f}) = G(\beta_1, \beta_2, \ldots, \beta_n)$. Clearly $\pi_{H}(\varprojlim \textbf{f}) \subseteq G(\beta_1, \beta_2, \ldots, \beta_n)$, so let $p = (p_{\beta_1}, p_{\beta_2}, \ldots, p_{\beta_n}) \in G(\beta_1, \beta_2, \ldots, \beta_n)$. We need to show that there exists $\textbf{y} \in \varprojlim \textbf{f}$ such that $y_{\beta_i} = p_{\beta_i}$ for $1 \le i \le n$.

\
If $M = \{\eta_{1}, \eta_{2}, \ldots, \eta_{k}\}$ is a finite subset of $D$ with $H \subseteq M$, let us define $p^{*}(M) = \{ \textbf{x} \in \prod_{\alpha \in D}X_{\alpha} \ | \ x_{\beta_{i}} = p_{\beta_i}$ for $1 \le i \le n$ and $x_{\sigma} \in f_{\sigma \tau}(x_{\tau})$ for all $\sigma, \tau \in M$ with $\sigma \preceq \tau\}$. We intend to show that $p^{*}(M)$ is non-empty and compact.

\
Let us construct an element $\textbf{x}$ of $p^{*}(M)$ as follows. First, let $x_{\beta_{i}} = p_{\beta_{i}}$ for $1 \le i \le n$. If $\{\eta_{1}, \eta_{2}, \ldots, \eta_{m}\}$ is the set of all members of $M$ with $\eta_{1}, \eta_{2}, \ldots, \eta_{m} \prec \beta_1$, then let $x_{\eta_{m}} \in f_{\eta_{m} \beta_{1}}(x_{\beta_1})$, and then continue inductively: for each integer $j$ with $m-1 \ge j \ge 1$, once $x_{\eta_{j+1}}$ has been defined, let $x_{\eta_{j}} \in f_{\eta_{j} \eta_{j+1}}(x_{\eta_{j+1}})$. Similarly, for a given integer $i$ with $1 \le i \le n-1$, suppose $\eta_r, \eta_{r+1}, \ldots, \eta_{s}$ are all the members of $M$ lying strictly between $\beta_{i}$ and $\beta_{i+1}$ in the ordering on $D$. Then since $x_{\beta_{i}} \in f_{\beta_{i} \beta_{i+1}}(x_{\beta_{i+1}})$ and $\textbf{f}$ is exact, there must exist some $x_{\eta_{s}} \in X_{\eta_{s}}$ so that $x_{\eta_{s}} \in f_{\eta_{s}\beta_{i+1}}(x_{\beta_{i+1}})$ and $x_{\beta_{i}} \in f_{\beta_{i}\eta_{s}}(x_{\eta_{s}})$. Using the same argument, we may continue in this way inductively: for each integer $j$ with $s-1 \ge j \ge r$, once $x_{\eta_{j+1}}$ has been defined, we select $x_{\eta_{j}} \in X_{\eta_{j}}$ so that $x_{\eta_{j}} \in f_{\eta_{j}\eta_{j+1}}(x_{\eta_{j+1}})$ and $x_{\beta_{i}} \in f_{\beta_{i}\eta_{j}}(x_{\eta_{j}})$. Finally, if $\eta_{t}, \eta_{t+1}, \ldots, \eta_{u}$ are all the members of $M$ that are strictly greater than $\beta_{n}$ in the ordering on $D$, then because each bonding function is surjective, we can choose some $x_{\eta_{t}} \in f^{-1}_{\beta_{n} \eta_{t}}(x_{\beta_n})$ and then proceed inductively: for each integer $j$ with $t+1 \le j \le u$, once $x_{\eta_{j-1}}$ has been defined, choose $x_{\eta_{j}} \in f^{-1}_{\eta_{j-1} \eta_{j}}(x_{\eta_{j-1}})$. For each $\alpha \in D - M$, we let $x_{\alpha} \in X_{\alpha}$ be chosen arbitrarily. Then $\textbf{x}$ satisfies $x_{\beta_i} = p_{\beta_{i}}$ for $1 \le i \le n$, and (by construction) whenever $\sigma, \tau \in M$ with $\sigma \preceq \tau$, we have $x_{\sigma} \in f_{\sigma \tau}(x_{\tau})$. So $\textbf{x} \in p^{*}(M)$. To see that $p^{*}(M)$ is closed, for $1 \le i \le n$, let $O_{\beta_i} = \pi_{\beta_{i}}^{-1}(X_{\beta_{i}} - \{p_{\beta_{i}}\})$. Each $O_{\beta_i}$ is open in $\prod_{\alpha \in D}X_{\alpha}$, so $G'(\eta_{1}, \eta_{2}, \ldots, \eta_{k}) - \bigcup_{1 \le i \le n}O_{\beta_i}$ is closed. However, it is not hard to see that $p^{*}(M) = G'(\eta_{1}, \eta_{2}, \ldots, \eta_{k}) - \bigcup_{1 \le i \le n}O_{\beta_i}$, so $p^{*}(M)$ is compact.

\
Note that, whenever $M$ and $N$ are two finite subsets of $D$ with $H \subseteq M$ and $H \subseteq N$, then $p^{*}(M \cup N) \subseteq p^{*}(M) \cap p^{*}(N)$. This means the collection $\Lambda = \{p^{*}(M) \ | \ M$ is a finite subset of $D$ with $H \subseteq M\}$ has the finite intersection property, and $\bigcap \Lambda$ is non-empty. Let $\textbf{y} \in \bigcap \Lambda$; we claim that $\textbf{y} \in \varprojlim \textbf{f}$. Let $\alpha_1, \alpha_2 \in D$ with $\alpha_1 \preceq \alpha_2$. Then since $\textbf{y} \in p^{*}(\{\alpha_1, \alpha_2\} \cup H)$, $y_{\alpha_{1}} \in f_{\alpha_{1} \alpha_{2}}(y_{\alpha_2})$. So indeed, $\textbf{y} \in \varprojlim \textbf{f}$, which means that there is an element $\textbf{y} \in \varprojlim \textbf{f}$ with $y_{\beta_i} = p_{\beta_i}$ for $1 \le i \le n$. We conclude that $G(\beta_1, \beta_2, \ldots, \beta_n) = \pi_{H}(\varprojlim \textbf{f})$.
\end{proof}

\begin{theorem} Suppose $\{X_{\alpha}, f_{\alpha \beta}, D\}$ is a system with non-empty compact Hausdorff factor spaces, u.s.c. bonding functions and a totally ordered index set $D$. Then $\varprojlim \textbf{f}$ is cordial if and only if, for all $\alpha, \beta \in D$ with $\alpha \preceq \beta$, $f_{\alpha \beta}$ is surjective.
\end{theorem}

\begin{proof} By Lemma \ref{cordial}, if each $f_{\alpha \beta}$ is surjective, $\varprojlim \textbf{f}$ is cordial. On the other hand, suppose $\varprojlim \textbf{f}$ is cordial and let $\alpha, \beta \in D$ with $\alpha \preceq \beta$. Let $y \in X_{\alpha}$; we intend to show that $f_{\alpha \beta}^{-1}(y)$ is non-empty. Because $\varprojlim \textbf{f}$ is cordial, $G(\alpha) = \pi_{\alpha}(\varprojlim \textbf{f})$. But $G(\alpha) = X_{\alpha}$. So, there exists $\textbf{x} \in \varprojlim \textbf{f}$ such that $x_{\alpha} = y$, and therefore $x_{\beta} \in f_{\alpha \beta}^{-1}(y)$. This means $f_{\alpha \beta}$ is surjective.
\end{proof}

It may also be interesting to note that, since $G(\alpha) = X_{\alpha}$ for each $\alpha$ in a directed index set $D$, when $\varprojlim \textbf{f}$ is cordial the corresponding system $\{X_{\alpha}, f_{\alpha \beta}, D\}$ must be consistent. However, as the previous example (where $f$ was the constant map $0$) shows, consistent systems do not always produce cordial inverse limits.

\section{\bf Idempotent u.s.c. Functions}

The main obstacle to constructing concrete examples of inverse limits indexed by some ``long'' totally ordered set (e.g., an uncountable limit ordinal such as $\omega_1$) is finding a large collection $\textbf{f}$ of bonding functions that is exact. One way around this difficulty is to consider generalized inverse limits with a single idempotent u.s.c. bonding function $f$; in this context, because $f^2 = f$, the collection of bonding functions is automatically exact. So, let us suppose $X$ is a non-empty compact Hausdorff space, $f: X \rightarrow 2^{X}$ is an idempotent u.s.c. bonding function, and $D$ is a totally ordered set. Then, as we have seen, $\{X, f, D\}$ is a system and the inverse limit $\varprojlim \textbf{f}$ with the single idempotent bonding function $f$ is defined.
\

Generalized inverse limits indexed by the positive integers with a single bonding function $f = f_{i \ i+1}$ for each $i \in \mathbb{Z}^+$ (and $f_{ij} = f_{i \ i+1} \circ f_{i+1 \ i+2} \circ \cdots \circ f_{j-1 \ j}$ whenever $i < j$) are commonly studied, and have given rise to many interesting problems and theorems. Therefore, we believe generalized inverse limits indexed by a totally ordered set $D$ with a single idempotent bonding function $f$ should be a natural (and, hopefully, fruitful) next step for the theory. In this section, we discuss some basic properties of idempotent u.s.c. functions and show how to construct some simple examples of such functions. The lemmas in this section will also be needed for some of the theorems and examples seen later in this paper.
\

\begin{lemma} \label{inverse} Suppose $X$ is a compact Hausdorff space and $f: X \rightarrow 2^{X}$ is surjective. Then $f$ is idempotent if and only if $f^{-1}$ is idempotent.
\end{lemma}

\begin{proof} Let $f$ be idempotent. $(y,x) \in Graph(f^{-1}) \Leftrightarrow (x,y) \in Graph(f) = Graph(f \circ f) \Leftrightarrow \exists z \in X$ such that $z \in f(x)$ and $y \in f(z) \Leftrightarrow \exists z \in X$ such that $x \in f^{-1}(z)$ and $z \in f^{-1}(y) \Leftrightarrow (y,x) \in Graph(f^{-1} \circ f^{-1})$. Thus, $f^{-1}$ is idempotent. The argument can easily be reversed by changing the roles of $f$ and $f^{-1}$.
\end{proof}

\begin{lemma} \label{idempotent lemma}
Let $X$ be a compact Hausdorff space and let $f: X \rightarrow 2^{X}$ be u.s.c. Then:
\

1. $f^2 = f$ if and only if, for each $A \subseteq X$ satisfying $f(x) = A$ for some $x \in X$, $f(A) = A$.
\

2. Suppose $f^2 = f$. If $f(x) = y$ for some $x, y \in X$, then $f(y) = y$.
\

3. If there is some $B \subseteq X$ so that, for all $x \in X$, either $f(x) = x$ or $f(x) = B$, then $f^2 = f$.
\

4. Suppose that, for some $A, B \subseteq X$ where $A \cap B = \emptyset$, we have $f(a) = X$ for each $a \in A$ and $f(x) = B$ whenever $x \not\in A$. Then $f^2 = f$.
\end{lemma}

\begin{proof} To prove statement 1, first note that, if $f^2 = f$ and $f(x) = A$, then $f(A) = f(f(x)) = f(x) = A$. To prove the other direction of the equivalence, we let $x \in X$, so $f(x) = A$ for some $A \subseteq X$. However, $f(A) = A$, so that $f^{2}(x) = f(f(x)) = f(A) = A = f(x)$. Applying the forward implication in statement 1 in the case where $A = \{y\}$ gives us statement 2. As for statement 3, if $f$ satisfies the given conditions, then whenever $x \in X$, either $f(x) = x$ (in which case, clearly $f^{2}(x) = x$) or $f(x) = B$ (in which case, for each $b \in B$, either $f(b) = b$ or $f(b) = B$, so that $f(B) = \bigcup_{b \in B} f(b) = B$, and therefore $f^{2}(x) = B$). Thus, $f^2 = f$. Finally, addressing statement 4, we note that when $x \in A$, $f(x) = X$ (so of course $f^{2}(x) = X$), and if $x \not\in A$, then $f(x) = B$, so that $f^{2}(x) = f(B)$. However, if $b \in B$, since $b \not\in A$, $f(b) = B$. Thus, $f(B) = B$, so that $f^{2}(x) = B$, and we conclude that $f^2 = f$. 
\end{proof}

In the special case where $X = [0,1]$, we also have the following. These statements are easy to verify directly and so we omit the proof. (Statement 3 below has been observed before by Ingram in \cite{ingram intro}.)

\begin{lemma} \label{idempotent lemma2} Let $f: [0,1] \rightarrow 2^{[0,1]}$ be u.s.c.
\

1. If, for all $x \in [0,1]$, either $f(x) = x$ or $f(x) = [0,x]$, then $f^2 = f$.
\

2. If, for all $x \in [0,1]$, either $f(x) = x$ or $f(x) = [x, 1]$, then $f^2 = f$.
\

3. If, for all $x \in [0,1]$, $f(x) = \{x, 1-x\}$, then $f^2 = f$.
\end{lemma}

Despite how restrictive the condition $f^2 = f$ may seem, the following lemma shows that one often has more freedom when constructing u.s.c. idempotent functions than it may first appear.

\begin{lemma} \label{K lemma} Let $a \in (0,1)$ and let $K$ be any closed subset of $[0,1]^2$ such that $K \subseteq ([0,a) \times (a,1]) \cup \{(a,a)\}$. Let $\Delta = \{(x,x) \in [0,1]^2 \ | \ x \in [0,1]\}$. Then $K \cup \Delta$ is the graph of an idempotent surjective u.s.c. function $f: [0,1] \rightarrow 2^{[0,1]}$.
\end{lemma}

\begin{proof} $K \cup \Delta$ is a closed subset of $[0,1]^2$ whose projections map onto both coordinates, so $K \cup \Delta$ is the graph of a surjective u.s.c. function $f: [0,1] \rightarrow 2^{[0,1]}$; it remains to check that $f$ is idempotent. If $a \le x \le 1$, then $f(x) = x$ and so $f^{2}(x) = x$ also. If $0 \le x < a$, then $f(x) = \{x\} \cup H_x$, where $H_x = \pi_{2}(K \cap (\{x\} \times (a,1]))$. In particular, $H_{x} \subseteq (a,1]$, and since $f|_{[a,1]}$ is the identity, $f(H_x) = H_x$. Thus, $f^{2}(x) = f(x) \cup f(H_x) = (\{x\} \cup H_x) \cup H_x = \{x\} \cup H_x = f(x)$. So, $f^2 = f$.
\end{proof}

\
We close this section with a remark: if $f: X \rightarrow X$ is continuous, idempotent and surjective, then by Lemma \ref{idempotent lemma} part 2, $f$ can only be the identity on $X$. However, a multitude of distinct idempotent, surjective, u.s.c. functions exist, which is an unexpected advantage to working in the general setting of u.s.c. functions.

\section{\bf Connectedness Results}

A considerable subset of the literature on generalized inverse limits has been devoted to finding necessary and/or sufficient conditions for $\varprojlim \textbf{f}$ to be connected. The difficulty of this problem is very well-known; still, researchers have produced various helpful results. In this section, we intend to reformulate some of these results in the context of generalized inverse limits indexed by a totally ordered set.

\
The key to proving the main theorems of this section will be the various ``$G$-sets'' introduced in Section 2 of this paper. So, the following lemma (which was originally proved by Ingram and Mahavier amid the proof of Theorem 124 in \cite{ingram mahavier}) will be useful:

\begin{lemma} \label{gamma lemma} (Ingram \& Mahavier) Let $\{X_{\alpha}, f_{\alpha \beta}, D\}$ be a system for which each factor space is non-empty compact Hausdorff, each bonding function is u.s.c., and $D$ is totally ordered. Let $\eta \in D$ be fixed, and let $\Gamma = \{G'(\beta_1, \beta_2, \ldots, \beta_n) \ | \ \{\beta_1, \beta_2, \ldots, \beta_n\}$ is a finite subset of $D$ with $\beta_{i} \preceq \eta$ for $1 \le i \le n\}$. Then $\bigcap \Gamma = G_\eta$.
\end{lemma}

\begin{proof} Let $\textbf{x} \in G_{\eta}$ and let $G'(\beta_1, \beta_2, \ldots, \beta_n)$ be an arbitrary member of $\Gamma$. Then because $\textbf{x} \in G_{\eta}$, we have $x_{\beta_i} \in f_{\beta_i \beta_j}(x_{\beta_j})$ for $1 \le i \le j \le n$. Thus, $\textbf{x} \in G'(\beta_1, \beta_2, \ldots, \beta_n)$, which implies $\textbf{x} \in \bigcap \Gamma$. On the other hand, suppose $\textbf{x} \in \bigcap \Gamma$. Then for any $\alpha \preceq \beta \preceq \eta$, because $\textbf{x} \in G'(\alpha, \beta)$, $x_{\alpha} \in f_{\alpha \beta}(x_{\beta})$. Thus, $\textbf{x} \in G_{\eta}$.
\end{proof}

\
The first main theorem of this section is an easy generalization of background Theorem \ref{background thm}. Note that the original theorem required $f_{\alpha \beta}$ to be Hausdorff continuum-valued (or $f_{\alpha \beta}(X_\beta)$ to be connected and $f^{-1}_{\alpha \beta}$ to be Hausdorff continuum-valued) for each $\alpha \preceq \beta$ in the index set $D$. This condition was imposed in order to force each $G(\beta_1, \beta_2, \ldots, \beta_n)$ to be connected, which was the key to the proof; however, there are more general circumstances in which $G(\beta_1, \beta_2, \ldots, \beta_n)$ also turns out to be connected, so the following theorem will be helpful. The proof requires only a minor adjustment to the original proofs of Theorems 124 and 125 given by Ingram and Mahavier in \cite{ingram mahavier}.

\begin{theorem} \label{finite connected}
Let $\{X_\alpha, f_{\alpha \beta}, D\}$ be a system for which each factor space is a Hausdorff continuum, each bonding function is u.s.c., and $D$ is a totally ordered set. Suppose that, for each finite subset $\{\beta_1, \beta_2, \ldots, \beta_n\}$ of $D$, $G(\beta_1, \beta_2, \ldots, \beta_n)$ is connected. Then $\varprojlim \textbf{f}$ is a Hausdorff continuum.
\end{theorem}

\begin{proof}
Let $\eta \in D$ be fixed; we intend to show that $G_{\eta}$ is a Hausdorff continuum. Let $H = \{\beta_1, \beta_2, \ldots, \beta_n\}$ be a finite subset of $D$. Then $G'(\beta_1, \beta_2, \ldots, \beta_n)$ is homeomorphic to $G(\beta_1, \beta_2, \ldots, \beta_n) \times \prod_{\alpha \in D - H} X_\alpha$, a product of Hausdorff continua; therefore, $G'(\beta_1, \beta_2, \ldots, \beta_n)$ is a Hausdorff continuum. Let $\Gamma = \{G'(\beta_1, \beta_2, \ldots, \beta_n) \ | \ \{\beta_1, \beta_2, \ldots, \beta_n\}$ is a finite subset of $D$ with $\beta_{i} \preceq \eta$ for $1 \le i \le n \}$. Next, note that whenever $\{\beta_1, \beta_2, \ldots, \beta_n\}$ and $\{\sigma_1, \sigma_2, \ldots, \sigma_s\}$ are finite subsets of $D$, if we let $\{\tau_1, \tau_2, \ldots, \tau_{r}\} = \{\beta_1, \beta_2, \ldots, \beta_n\} \cup \{\sigma_1, \sigma_2, \ldots, \sigma_s\}$, then we may conclude that $G'(\tau_1, \tau_2, \ldots, \tau_{r}) \subseteq G'(\beta_1, \beta_2, \ldots, \beta_n) \cap G'(\sigma_1, \sigma_2, \ldots, \sigma_s)$. Thus, $\Gamma$ is a collection of Hausdorff continua with the property that the intersection of any finite subcollection of members of $\Gamma$ contains another member of $\Gamma$. It follows that $\bigcap \Gamma$ is a Hausdorff continuum. However, by Lemma \ref{gamma lemma}, $\bigcap \Gamma = G_{\eta}$. $G_{\eta}$ is therefore a Hausdorff continuum for each $\eta \in D$.

\
Finally, it is easy to see that $\varprojlim \textbf{f} = \bigcap_{\eta \in D}G_{\eta}$, but $\{G_{\eta} \ | \ \eta \in D\}$ is a nested collection of Hausdorff continua, so $\varprojlim \textbf{f}$ is also a Hausdorff continuum.
\end{proof}

We may apply the above theorem to obtain a characterization of connectedness in inverse limits with surjective bonding functions:

\begin{theorem} \label{finite connected iff}
Let $\{X_\alpha, f_{\alpha \beta}, D\}$ be a system for which each factor space is a Hausdorff continuum, each bonding function is u.s.c. and surjective, and $D$ is a totally ordered set. Then $\varprojlim \textbf{f}$ is a Hausdorff continuum iff for each finite subset $\{\beta_1, \beta_2, \ldots, \beta_n\}$ of $D$, $G(\beta_1, \beta_2, \ldots, \beta_n)$ is connected.
\end{theorem}

\begin{proof} Suppose that, for each finite subset $\{\beta_1, \beta_2, \ldots, \beta_n\}$ of $D$, the set $G(\beta_1, \beta_2, \ldots, \beta_n)$ is connected.  By Theorem \ref{finite connected}, we may conclude that $\varprojlim \textbf{f}$ is a Hausdorff continuum. On the other hand, suppose  $\varprojlim \textbf{f}$ is a Hausdorff continuum. Since the bonding functions are surjective,  $\varprojlim \textbf{f}$ is cordial, and therefore for any finite subset $H = \{\beta_1, \beta_2, \ldots, \beta_n\}$ of $D$, $\pi_{H}(\varprojlim \textbf{f}) = G(\beta_1, \beta_2, \ldots, \beta_n)$, which implies that $G(\beta_1, \beta_2, \ldots, \beta_n)$ is connected.
\end{proof}

A similar characterization, this time in terms of $G_{\eta}$, is also worth noting. (This result was inspired by Theorem 2.1 in \cite{ingram intro}.)

\begin{theorem} Let $\{X_\alpha, f_{\alpha \beta}, D\}$ be a system for which each factor space is a Hausdorff continuum, each bonding function is u.s.c. and surjective, and $D$ is a totally ordered set. Then $\varprojlim \textbf{f}$ is a Hausdorff continuum iff $G_{\eta}$ is connected for each $\eta \in D$.
\end{theorem}

\begin{proof} Since $\varprojlim \textbf{f} = \bigcap_{\eta \in D} G_{\eta}$, if $G_{\eta}$ is connected for each $\eta \in D$, then $\varprojlim \textbf{f}$ is the intersection of a nested collection of Hausdorff continua and is therefore a Hausdorff continuum. On the other hand, suppose $\varprojlim \textbf{f}$ is a Hausdorff continuum and $\eta \in D$. By Theorem \ref{finite connected iff}, when $\{\beta_1, \beta_2, \ldots, \beta_{n}\}$ is a finite subset of $D$, $G(\beta_1, \beta_2, \ldots, \beta_{n})$ is a Hausdorff continuum, which would mean that $G'(\beta_1, \beta_2, \ldots, \beta_{n})$ is also a Hausdorff continuum. So, if $\Gamma = \{G'(\beta_1, \beta_2, \ldots, \beta_n) \ | \ \{\beta_1, \beta_2, \ldots, \beta_n\}$ is a finite subset of $D$ with $\beta_{i} \preceq \eta$ for $1 \le i \le n \}$, then by the same argument given in the proof of Theorem \ref{finite connected}, $\bigcap \Gamma$ is a Hausdorff continuum. By Lemma \ref{gamma lemma}, $\bigcap \Gamma = G_{\eta}$.
\end{proof}

In the last two theorems, the hypothesis that each bonding function is surjective is necessary. To see this, consider the following example (which can also be found as Example 1.8 in \cite{ingram intro}). Let $f: [0,1] \rightarrow 2^{[0,1]}$ be given by $f(x) = 0$ for $0 \le x < 1$ and $f(1) = \{0, 1/2\}$, let the index set be the positive integers, and let $f_{i \ i+1} = f$ for each positive integer $i$. Then the inverse limit $\varprojlim \textbf{f}$ is the singleton $\{(0,0,0,\ldots)\}$, whereas $G(1,2)$, being homeomorphic to the graph of $f$, is not connected. $G_{2}$ is not connected for similar reasons. The source of the trouble here is that, because the bonding functions are not surjective, $\varprojlim \textbf{f}$ is not cordial, and so, $G(1,2)$ need not equal $\pi_{\{1,2\}}(\varprojlim \textbf{f})$. This example also shows that, in Theorem 2.1 of \cite{ingram intro}, the bonding functions should have been assumed to be surjective in order for $\varprojlim \textbf{f}$ being a continuum to imply that each $G_n$ is connected. (The converse, which is more commonly used, is true as it stands, without having to assume surjectivity.) Ingram has asked the author to make the readers of this paper aware of the error.

\
Next, we give a simple sufficient condition for non-connectedness in an inverse limit with surjective bonding functions. (This result generalizes an observation by Nall in \cite{nall connected}.) Once again, the previous example shows that the bonding functions in this theorem must be surjective.

\begin{theorem} Let $\{X_{\alpha}, f_{\alpha \beta}, D\}$ be a system for which each factor space is non-empty, compact and Hausdorff, each bonding function is u.s.c. and surjective, and $D$ is totally ordered. If, for some $\alpha \preceq \beta \in D$, the graph of $f_{\alpha \beta}$ is not connected, then $\varprojlim \textbf{f}$ is not connected.
\end{theorem}  

\begin{proof} Suppose $\varprojlim \textbf{f}$ is connected. Then because it is cordial (by Lemma \ref{cordial}), $\pi_{\{\alpha, \beta\}}(\varprojlim \textbf{f}) = G(\alpha, \beta)$ is connected for every $\alpha \preceq \beta \in D$. Since, for each $\alpha \preceq \beta \in D$, $G(\alpha, \beta)$ is homeomorphic to $Graph(f_{\alpha \beta})$, the proof is complete.
\end{proof}

\
The reader should be warned that the converse of the previous theorem is not true, as illustrated by an example from Jonathan Meddaugh that is given in \cite{ingramnonconn}.

\
For the remainder of this section, we will focus on generalized inverse limits with a single idempotent surjective u.s.c. bonding function. In \cite{nall connected}, Van Nall presents a variety of theorems that give information about the connectedness of a generalized inverse limit (indexed by the positive integers) where there is some u.s.c. function $f$ so that $f_{i \  i+1} = f$ for each $i > 0$. Using techniques from the original proofs by Nall, we will reformulate two of those theorems in the setting of an inverse limit (indexed by some totally ordered set $D$) with a single idempotent u.s.c. bonding function $f$.

\
To prepare for the first of these theorems, we give a lemma that restates a key detail from Nall's original Theorem 3.1 in \cite{nall connected}, but in the context of this paper. (For the proof of this lemma, we refer the reader to Nall's paper.) Suppose $X$ is compact Hausdorff, $f: X \rightarrow 2^{X}$ is u.s.c., $\{\beta_1, \beta_2, \ldots, \beta_n\}$ (with $n \ge 2$) is a finite subset of a totally ordered set $D$, and $X_{\beta_i} = X$ for $1 \le i \le n$. Then let us define $K(\beta_1, \beta_2, \ldots, \beta_n) = \{(x_{\beta_{1}}, x_{\beta_{2}}, \ldots, x_{\beta_n}) \in \prod_{1 \le i \le n} X_{\beta_{i}} \ | \ x_{\beta_{i}} \in f(x_{\beta_{i+1}})$ for $1 \le i \le n-1\}$.

\begin{lemma} \label{Kn lemma} (Nall) Suppose $X$ is a metric continuum and $f: X \rightarrow 2^{X}$ is a surjective u.s.c. function whose graph is connected and is the union of a collection of u.s.c. functions, each of which has domain $X$ and maps into $C(X)$. Suppose also that $\{\beta_1, \beta_2, \ldots, \beta_n\}$ (with $n \ge 2$) is a finite subset of a totally ordered set $D$, and $X_{\beta_i} = X$ for $1 \le i \le n$. Then $K(\beta_1, \beta_2, \ldots, \beta_n)$ is a continuum for each integer $n \ge 2$.
\end{lemma}

Nall proved the above result in the case of $K(1,2,\ldots,n)$, but of course the same result holds if we replace the natural numbers $1,2,\ldots,n$ with other symbols $\beta_1, \beta_2, \ldots, \beta_n$ from a totally ordered set $D$. Also, not surprisingly, the set $K(\beta_1, \beta_2, \ldots, \beta_n)$ can be rewritten using the ``$G$'' notation of this paper:

\begin{lemma} \label{sets equal} Let $X$ be a non-empty compact Hausdorff space, and let $f: X \rightarrow 2^{X}$ be an idempotent surjective u.s.c. function. Suppose $\{X, f, D\}$ is a system with the single idempotent bonding function $f$ and a totally ordered index set $D$. If $\{\beta_1, \beta_2, \ldots, \beta_n\}$ (with $n \ge 2$) is a finite subset of $D$, then $K(\beta_1, \beta_2, \ldots, \beta_n) = G(\beta_1, \beta_2, \ldots, \beta_n)$.
\end{lemma}

\begin{proof} If $(x_{\beta_1}, x_{\beta_2}, \ldots, x_{\beta_n}) \in G(\beta_1, \beta_2, \ldots, \beta_n)$, then because we have $x_{\beta_i} \in f_{\beta_{i} \beta_{i+1}}(x_{\beta_{i+1}}) = f(x_{\beta_{i+1}})$ for $1 \le i \le n-1$, it follows that $(x_{\beta_1}, x_{\beta_2}, \ldots, x_{\beta_n}) \in K(\beta_1, \beta_2, \ldots, \beta_n)$. Now let $(x_{\beta_1}, x_{\beta_2}, \ldots, x_{\beta_n}) \in K(\beta_1, \beta_2, \ldots, \beta_n)$. Clearly $x_{\beta_i} \in f_{\beta_{i}\beta_{i+1}}(x_{\beta_{i+1}})$ for $1 \le i < n$. So now, for a fixed $i$ with $1 \le i < n$, we will show inductively that $x_{\beta_i} \in f_{\beta_{i}\beta_{j}}(x_{\beta_j})$ when $i < j \le n$. Suppose we have shown $x_{\beta_i} \in  f_{\beta_{i}\beta_{k}}(x_{\beta_k})$ for some $k$ with $i+1 \le k < n$. Then since $x_{\beta_{k}} \in f_{\beta_{k} \beta_{k+1}}(x_{\beta_{k+1}})$, we know $x_{\beta_{i}} \in f_{\beta_{i}\beta_{k}} \circ f_{\beta_{k}\beta_{k+1}}(x_{\beta_{k+1}})$. However, $f_{\beta_{i}\beta_{k}} \circ f_{\beta_{k}\beta_{k+1}} = f \circ f = f = f_{\beta_{i} \beta_{k+1}}$, so we have $x_{\beta_{i}} \in f_{\beta_{i}\beta_{k+1}}(x_{\beta_{k+1}})$. Thus, $x_{\beta_{i}} \in f_{\beta_{i}\beta_{j}}(x_{\beta_{j}})$ for all $1 \le i \le j \le n$, and $(x_{\beta_1}, x_{\beta_2}, \ldots, x_{\beta_n}) \in G(\beta_1, \beta_2, \ldots, \beta_n)$.
\end{proof}

\begin{theorem} \label{nall lemma} Suppose $X$ is a metric continuum and $f: X \rightarrow 2^{X}$ is an idempotent surjective u.s.c. function whose graph is connected and is the union of a collection of u.s.c. functions, each of which has domain $X$ and maps into $C(X)$. Let $\{X, f, D\}$ be a system with the single idempotent bonding function $f$ and a totally ordered index set $D$. Then the inverse limit $\varprojlim \textbf{f}$ of the system $\{X, f, D\}$ is a Hausdorff continuum.
\end{theorem}

\begin{proof} By Theorem \ref{finite connected}, it suffices to show that for each finite subset $\{\beta_1, \beta_2, \ldots, \beta_n\}$ of $D$, $G(\beta_1, \beta_2, \ldots, \beta_n)$ is connected. Of course, $G(\beta_1) = X$ is connected, so assume $\{\beta_1, \beta_2, \ldots, \beta_n\}$ is a finite subset of $D$ with $n \ge 2$. By Lemma \ref{sets equal}, $G(\beta_1, \beta_2, \ldots, \beta_n) = K(\beta_1, \beta_2, \ldots, \beta_n)$. However, by Lemma \ref{Kn lemma}, $K(\beta_1, \beta_2, \ldots, \beta_n)$ is a continuum, so the proof is complete.
\end{proof}

Lemma \ref{inverse} stated that if $f$ is u.s.c., idempotent and surjective, then so is $f^{-1}$. Thus, assuming $X$ is a non-empty compact Hausdorff space and $f: X \rightarrow 2^{X}$ is u.s.c., idempotent, and surjective, then if $D$ is totally ordered, not only is the inverse limit $\varprojlim \textbf{f}$ of the system $\{X, f, D\}$ defined, but also the inverse limit $\varprojlim \textbf{f}^{\textbf{--1}}$ of the system $\{X, f^{-1}, D\}$. This is the basis for the last theorem of this section, which retains the flavor of Theorem 3.3 in \cite{nall connected}.

\begin{theorem} Suppose $X$ is a Hausdorff continuum, $f: X \rightarrow 2^{X}$ is an idempotent surjective u.s.c. function, and $D$ is a totally ordered set. Let $\varprojlim \textbf{f}$ be the inverse limit of the system $\{X, f, D\}$, and let $\varprojlim \textbf{f}^{\textbf{--1}}$ be the inverse limit of the system $\{X, f^{-1}, D\}$. Then $\varprojlim \textbf{f}$ is a Hausdorff continuum if and only if $\varprojlim\textbf{f}^{\textbf{--1}}$ is a Hausdorff continuum.

\end{theorem}

\begin{proof} We employ a similar proof technique as the one used by Nall. Let $\{\beta_1, \beta_2, \ldots, \beta_{n}\}$ be a finite subset of $D$. If $\varprojlim \textbf{f}$ is a Hausdorff continuum, then because it is cordial, we may conclude that the set $G_{f}(\beta_1, \beta_2, \ldots, \beta_{n})$ is a Hausdorff continuum. Also, $(x_{\beta_{1}}, x_{\beta_{2}}, \ldots, x_{\beta_{n}}) \in G_{f}(\beta_1, \beta_2, \ldots, \beta_{n})$ if and only if $(x_{\beta_{n}}, x_{\beta_{n-1}}, \ldots, x_{\beta_{1}}) \in G_{f^{-1}}(\beta_1, \beta_2, \ldots, \beta_{n})$. (Justification: $x_{\beta_i} \in f(x_{\beta_{j}})$ for each $1 \le i < j \le n$ if and only if $x_{\beta_j} \in f^{-1}(x_{\beta_i})$ for each $1 \le i < j \le n$.) Thus, the map that reverses the order of the sequence $(x_{\beta_{1}}, x_{\beta_2}, \ldots, x_{\beta_{n}})$ is a homeomorphism from $G_{f}(\beta_1, \beta_2, \ldots, \beta_{n})$ to $G_{f^{-1}}(\beta_1, \beta_2, \ldots, \beta_{n})$, and therefore, $G_{f^{-1}}(\beta_1, \beta_2, \ldots, \beta_{n})$ is a Hausdorff continuum. We apply Theorem \ref{finite connected} to conclude that $\varprojlim \textbf{f}^{\textbf{--1}}$ is a Hausdorff continuum. The argument is easily reversed to obtain the other direction of the equivalence.
\end{proof}

\section{\bf Examples}

As we will show, in each of the following examples, the given surjective u.s.c. function $f: [0,1] \rightarrow 2^{[0,1]}$ is idempotent. Thus, if $D$ is a totally ordered set, the inverse limit $\varprojlim \textbf{f}$ of the system $\{X, f, D\}$ with the single bonding function $f$ is defined. We will use a limit ordinal $\gamma \ge \omega$ for our index set $D$ in these examples. (Recall that an ordinal $\gamma$ is equal to the set of its predecessors. Although the index sets of inverse limits typically start at $1$ because the positive integers are so often used, when we use ordinals our index set will start with $0$. For more background material on ordinals, see, e.g., \cite{Kunen}.) Of course, if $\gamma = \omega$, then $\varprojlim \textbf{f}$ is homeomorphic to the usual generalized inverse limit indexed by the positive integers. Versions of the first four examples have been studied previously by others (e.g., Ingram in \cite{ingram intro}) in that setting; however, as we shall see below, some striking changes can occur when larger limit ordinals are chosen for the index set---especially $\gamma \ge \omega_1$.

\begin{example} Let $f:[0,1] \rightarrow C([0,1])$ be given by $f(0) = [0,1]$ and $f(x) = x$ for each $x \ne 0$. ($f$ is idempotent by Lemma \ref{idempotent lemma}, part 3.) The inverse limit $\varprojlim \textbf{f}$ of the system $\{[0,1], f, \gamma\}$ is a fan that is the union of one arc for each ordinal $\beta$ with $1 \le \beta \le \gamma$, all intersecting at the vertex $(0,0,0,\ldots)$.
\end{example}

\begin{proof} For $1 \le \beta \le \gamma$, let $A_{\beta} = \{ \textbf{x} \in \prod_{\alpha < \gamma} [0,1] \ | \ x_{\alpha} = x_{0}$ if $\alpha < \beta$ and $x_{\alpha} = 0$ if $\beta \le \alpha < \gamma\}$. Then for $1 \le \beta \le \gamma$, $A_{\beta}$ is an arc and $\varprojlim \textbf{f} = \bigcup_{1 \le \beta \le \gamma} A_{\beta}$. Note that $\bigcap_{1 \le \beta \le \gamma} A_{\beta} = \{(0,0,0, \ldots)\}$. 
\end{proof}

\begin{example} \label{long line} Let $f:[0,1] \rightarrow C([0,1])$ be given by $f(0) = [0,1]$ and $f(x) = 1$ for each $x \ne 0$. ($f$ is idempotent by Lemma \ref{idempotent lemma}, part 4.) The inverse limit $\varprojlim \textbf{f}$ of the system $\{[0,1], f, \gamma\}$ is homeomorphic to the set $(\gamma \times [0,1)) \cup \{(\gamma, 0)\}$ with the lexicographic order topology. (If $\gamma = \omega_1$, this space is the traditional compactified long line.)
\end{example}

\begin{proof} Let $Y$ be the space $(\gamma \times [0,1)) \cup \{(\gamma, 0)\}$ with the lexicographic order topology, and denote the points $(0,0,0,\ldots)$ and $(1,1,1,\ldots)$ of $\varprojlim \textbf{f}$ by $\textbf{0}$ and $\textbf{1}$, respectively. Let $h: \varprojlim \textbf{f} \rightarrow Y$ be given as follows: suppose $\textbf{x} = (x_{\alpha})_{\alpha < \gamma} \in \varprojlim \textbf{f}$. If $\textbf{x} = \textbf{1}$, then $h(\textbf{x}) = (\gamma, 0)$; if $\textbf{x} \ne \textbf{1}$, then $h(\textbf{x}) = (\beta, x_{\beta})$, where $\beta$ is the least ordinal $< \gamma$ such that $x_{\beta} \ne 1$. We intend to show that $h$ is a homeomorphism.

\
Let $\textbf{x}, \textbf{y} \in \varprojlim \textbf{f}$ with $\textbf{x} \ne \textbf{y}$. If exactly one of \textbf{x} or \textbf{y} is \textbf{1}, then clearly $h(\textbf{x}) \ne h(\textbf{y})$. So, suppose both \textbf{x} and \textbf{y} are not \textbf{1}. Then $h(\textbf{x}) = (\beta_{1}, x_{\beta_{1}})$ and $h(\textbf{y}) = (\beta_{2}, x_{\beta_{2}})$ for some $\beta_1 , \beta_2 < \gamma$. If $\beta_{1} \ne \beta_{2}$, then of course $h(\textbf{x}) \ne h(\textbf{y})$. If $\beta_1 = \beta_2$, then by the way $f$ and $h$ were defined, $x_{\alpha} = y_{\alpha} = 1$ for all $\alpha < \beta_1$, and $x_{\alpha} = y_\alpha = 0$ for all $\alpha > \beta_1$. However, $\textbf{x} \ne \textbf{y}$, and that forces $x_{\beta_1} \ne y_{\beta_1}$, which means $h(\textbf{x}) \ne h(\textbf{y})$. So, $h$ is one-to-one. To show $h$ is onto, we recall that $h(\textbf{1}) = (\gamma, 0)$, so let $(\beta, t) \in Y - \{(\gamma, 0)\}$. Then if $\textbf{x} = (x_{\alpha})_{\alpha < \gamma}$ is the element of the inverse limit that satisfies $x_{\alpha} = 1$ for all $\alpha < \beta$, $x_{\beta} = t$, and $x_{\alpha} = 0$ for all $\alpha > \beta$, we have $h(\textbf{x}) = (\beta, t)$. So, $h$ is indeed onto.

\
Finally, to show $h$ is continuous, we let $\textbf{x} \in \varprojlim \textbf{f}$ and let $V$ be an open set in $Y$ containing $h(\textbf{x})$. We consider the subcase where $h(\textbf{x}) = (\beta, x_{\beta})$, where $\beta < \gamma$ and $0 < x_{\beta} < 1$. Then $V$ contains an open interval of form $((\beta, s), (\beta, t))$ where $0 < s < x_{\beta} < t < 1$. Let $U = (\prod_{\alpha < \gamma} U_{\alpha}) \cap \varprojlim \textbf{f}$, where $U_{\alpha} = [0,1]$ for all $\alpha \ne \beta$ and $U_{\beta} = (s,t)$. Then $U$ is open in $\varprojlim \textbf{f}$ and $\textbf{x} \in U$. To see that $h(U) \subseteq V$, let $\textbf{y} \in U$ and note that $y_{\beta} \in (s,t)$. Since $0 < s < y_{\beta} < t < 1$, we may conclude $y_{\alpha} = 1$ for all $\alpha < \beta$, and that means $h(\textbf{y}) = (\beta, y_{\beta})$, which lies in $((\beta, s), (\beta, t)) \subseteq V$. The remaining subcases (i.e., when $h(\textbf{x}) = (\beta,0)$ for some $0 \le \beta \le \gamma$), though slightly more complicated, are similar and will be left to the reader.

\
Thus, $h$ is one-to-one, onto, and continuous, so $h^{-1}$ is also continuous and $h$ is a homeomorphism.
\end{proof}

A similar argument can be used for the following example, so we omit the proof.

\begin{example} Let $f:[0,1] \rightarrow C([0,1])$ be given by $f(0) = [0, 1/2], f(1) = [1/2, 1]$, and $f(x) = 1/2$ whenever $0 < x < 1$. ($f$ can be shown to be idempotent directly, or by applying Lemma \ref{idempotent lemma}, part 1.) The inverse limit $\varprojlim \textbf{f}$ of the system $\{[0,1], f, \gamma\}$ is an arc homeomorphic to the union of two copies of the space produced in Example \ref{long line} intersecting at the common compactification point $(1/2, 1/2, 1/2, \ldots)$.
\end{example}

It may be interesting to note that in the previous two examples, if $\gamma$ is chosen to be any limit ordinal $< \omega_1$, then the inverse limit is simply a metric arc. Only once $\gamma = \omega_1$ is chosen do we get a non-metric arc.

\begin{example} Let $f: [0,1] \rightarrow 2^{[0,1]}$ be given by $f(x) = \{x, 1-x\}$ for each $x \in [0,1]$. ($f$ is idempotent by Lemma \ref{idempotent lemma2}, part 3.) The inverse limit $\varprojlim \textbf{f}$ of the system $\{[0,1], f, \gamma\}$ is a cone over the set $\{0,1\}^{\gamma}$ with vertex $(1/2, 1/2, 1/2, \ldots)$.
\end{example}

\begin{proof} For a given $\textbf{y} \in \{t,1-t\}^{\gamma}$ and $a \in [0,1]$, let $\textbf{y}(a)$ denote the sequence $\textbf{y}$ with each $t$ replaced by $a$. For a given $\textbf{y} \in \{t,1-t\}^{\gamma}$, let $A_{\textbf{y}} = \{ \textbf{y}(a) \ | \ a \in [0,1]\}$. Then $\varprojlim \textbf{f} = \bigcup_{\textbf{y} \in \{t,1-t\}^{\gamma}} A_{\textbf{y}}$. Note that, for each $\textbf{y}$, $A_{\textbf{y}}$ is an arc from $\textbf{y}(0)$ to $\textbf{y}(1)$ and whenever $\textbf{y}, \textbf{z} \in \{t,1-t\}^{\gamma}$ with $\textbf{y} \ne \textbf{z}$, $A_{\textbf{y}} \cap A_{\textbf{z}} = \{(1/2, 1/2, 1/2, \ldots)\}$.
\end{proof}

\begin{example} Let $f: [0,1] \rightarrow 2^{[0,1]}$ satisfy the hypothesis of Lemma \ref{K lemma} (so $f$ is idempotent), with the additional requirement that $K$ is a continuum containing the point $(a,a)$. Then the inverse limit $\varprojlim \textbf{f}$ of the system $\{[0,1], f, \gamma\}$ is a Hausdorff continuum. (This continuum contains a fan of copies of $K$, with one copy of $K$ for each ordinal $1 \le \beta < \gamma$.)
\end{example}

\begin{proof} For $1 \le \beta < \gamma$, let $A_{\beta} = \{ \textbf{x} \in \varprojlim \textbf{f} \ | \ \exists (s,t) \in K$ with $x_{\alpha} = t \ \forall \alpha < \beta$ and $x_{\alpha} = s \ \forall \alpha \ge \beta\}$. Let $B = \{ \textbf{x} \in \varprojlim \textbf{f} \ | \ \exists t \in [0,1] $ such that $x_{\alpha} = t \ \forall \alpha < \gamma\}$. Then $\varprojlim \textbf{f} = (\bigcup_{1 \le \beta < \gamma}A_{\beta}) \cup B$. Each $A_{\beta}$ can be seen to be homeomorphic to $K$, a continuum, and $B$ is an arc. Since the sequence $\textbf{a} = (a)_{\alpha < \gamma}$ is an element of each $A_{\beta}$ as well as $B$, $\varprojlim \textbf{f}$ is connected.
\end{proof}

\section{\bf Conclusion}

Virtually any question that has been stated for inverse limits indexed by the positive integers has an analogue for inverse limits indexed by totally ordered sets, so there are ample opportunities for further research. A number of problems (stated mainly for inverse limits indexed by the positive integers) can be found in \cite{ingramproblemspaper} and in Chapter 6 of \cite{ingram intro}. One question of interest would be the following:

\begin{question} Choose some totally ordered set $D$. What continua are homeomorphic to an inverse limit with a single u.s.c. idempotent surjective bonding function $f: [0,1] \rightarrow 2^{[0,1]}$ with index set $D$?
\end{question}

Choosing index sets other than the positive integers can have surprising effects on this problem. For example, Patrick Vernon used the integers as his index set to produce a 2-cell in \cite{vernon}, whereas Van Nall proved (in \cite{nall 2cell}) that a 2-cell cannot be produced if the index set consists only of the positive integers. (The bonding function used in Vernon's example was not idempotent, however.)
\

Let us also state a much more open-ended question: 

\begin{question} If $\varprojlim \textbf{f}$ is the inverse limit of a system $\{[0,1], f, D\}$ with a single u.s.c. idempotent surjective bonding function $f: [0,1] \rightarrow 2^{[0,1]}$ and totally ordered index set $D$, what can be said of $\varprojlim \textbf{f}$?
\end{question}

When the index set $D$ is large, finding collections of u.s.c. functions $\textbf{f}$ that are exact (without being trivial, e.g., by making almost all of the bonding functions be the identity) remains a difficult problem. Using a single idempotent bonding function is only one possible solution. The collection of continuous bonding functions used by Michel Smith in \cite{smith} is exact, although the factor spaces become increasingly complicated as one moves deeper into the index set $D$.

\begin{question} What other techniques are there for generating non-trivial collections of u.s.c. (surjective) functions that are exact?
\end{question}

The examples section showed how strongly the choice of index set $D$ can affect the properties of the inverse limit space. So, we close with another open-ended question:

\begin{question} If the index set $D$ has a given property $P$, under what conditions (and to what degree) does that impact the properties of $\varprojlim \textbf{f}$?
\end{question}

\bibliographystyle{plain}

\end{document}